\documentclass{amsart}
\usepackage{amsfonts, amsbsy, amsmath, amssymb}

\makeatletter
\@namedef{subjclassname@2010}{%
  \textup{2020} Mathematics Subject Classification}
\makeatother

\ifx\fiverm\undefined 
  \newfont\fiverm{cmr5} 
\fi 

\input prepictex.tex
\input pictex.tex
\input postpictex.tex

\newtheorem{thm}{Theorem}[section]

\newtheorem{exmp}[thm]{Example}

\newtheorem{rmk}[thm]{Remark}

\newtheorem{ques}[thm]{Question}
\newtheorem{res}[thm]{Result}
\numberwithin{equation}{section}

\theoremstyle{definition}

\allowdisplaybreaks

\newcommand{\f}{\Bbb F}

\begin{document}

\title{New Results on Permutation Binomials of Finite Fields}

\author[Xiang-dong Hou]{Xiang-dong Hou}
\address{Department of Mathematics and Statistics,
University of South Florida, Tampa, FL 33620}
\email{xhou@usf.edu}

\author[Vincenzo Pallozzi Lavorante]{Vincenzo Pallozzi Lavorante}
\address{Universit\`a degli Studi di Modena e Reggio Emilia, Italy}
\email{vincenzo.pallozzilavorante@unimore.it}



\begin{abstract}
After a brief review of existing results on permutation binomials of finite fields, we introduce the notion of equivalence among permutation binomials (PBs) and describe how to bring a PB to its canonical form under equivalence. We then focus on PBs of $\Bbb F_{q^2}$ of the form $X^n(X^{d(q-1)}+a)$, where $n$ and $d$ are positive integers and $a\in\Bbb F_{q^2}^*$. Our contributions include two nonexistence results: (1) If $q$ is even and sufficiently large and $a^{q+1}\ne 1$, then $X^n(X^{3(q-1)}+a)$ is not a PB of $\Bbb F_{q^2}$. (2) If $2\le d\mid q+1$, $q$ is sufficiently large and $a^{q+1}\ne 1$, then $X^n(X^{d(q-1)}+a)$ is not a PB of $\Bbb F_{q^2}$ under certain additional conditions. (1) partially confirms a recent conjecture by Tu et al. (2) is an extension of a previous result with $n=1$.
\end{abstract}

\maketitle

\section{Introduction}

Let $\f_q$ be the finite field with $q$ elements and $\f_q^*$ be its multiplicative group. A polynomial $f\in\f_q[X]$ is called a permutation polynomial (PP) of $\f_q$ if it induces a permutation of $\f_q$. A permutation binomial (PB) of $\f_q$ is a PP of the form $aX^m+bX^n$, where $a,b\in\f_q^*$, $m\not\equiv 0$, $n\not\equiv 0$ and $m\not\equiv n\pmod{q-1}$. Permutation binomials are an active topic that has attracted much attention. We refer the reader to \cite{Hou-CM-2015} for a survey on PBs and to \cite{Hou-FFA-2015a} for a survey on PPs. Permutation binomials are complex objects; in general, one can not expect a simple criterion on the parameters $q,m,n,a,b$ for $aX^m+bX^n$ to be a PB of $\f_q$. In this paper, we focus on PBs of $\f_{q^e}$ of the form
\begin{equation}\label{f}
f_{q,e,n,d,a}(X)=X^n(X^{d(q-1)}+a)\in\f_{q^e}[X],
\end{equation}
where $n,d\in\Bbb Z^+$, $n\not\equiv 0$, $d(q-1)\not\equiv 0$, $n+d(q-1)\not\equiv 0\pmod{q^e-1}$, and $a\in\f_{q^e}^*$. Here is an overview of current knowledge on such PBs.

\begin{res}[{\cite[Corollary~5.3]{Zieve-arXiv1310.0776}}]\label{R1.1}
Assume $e=2$ and $a^{q+1}=1$. Then $f_{q,2,n,d,a}=X^n(X^{d(q-1)}+a)$ is a PB of $\f_{q^2}$ if and only if $\text{\rm gcd}(n,q-1)=1$, $\text{\rm gcd}(n-d,q+1)=1$ and $(-a)^{(q+1)/\text{\rm gcd}(q+1,d)}\ne 1$.
\end{res}

\begin{res}[{\cite[Theorem~A]{Hou-FFA-2015b}}]\label{R1.2}
Assume $e=2$, $n=1$, $d=2$ and $a^{q+1}\ne 1$. Then $f_{q,2,1,2,a}=X(X^{2(q-1)}+a)$ is a PB of $\f_{q^2}$ if and only if $q$ is odd and $(-a)^{(q+1)/2}=3$.
\end{res}

\begin{res}[{\cite[Theorem~1.1]{Hou-FFA-2018}}]\label{R1.3}
Assume $e=2$, $n=1$, $d>2$, $a^{q+1}\ne 1$, and $q$ is large relative to $d$. Then $f_{q,2,1,d,a}=X(X^{d(q-1)}+a)$ is not a PB of $\f_{q^2}$.
\end{res}

\begin{res}[\cite{Lappano-thesis, Lappano-pc}]\label{R1.4}
Assume $e=2$, $n=3$, $d=2$ and $a^{q+1}\ne 1$. Then $f_{q,2,3,2,a}=X^3(X^{2(q-1)}+a)$ is a PB of $\f_{q^2}$ if and only if $q$ is odd, $q\equiv -1\pmod 3$ and $(-a)^{(q+1)/2}=1/3$. 
\end{res}

\begin{res}[{\cite[Theorem~1]{Tu-Zeng-Jiang-Li-DCC-2021}}]\label{R1.5}
Assume $e=2$, $q=2^{2m}$ and $d=3$. Then $f_{q,2,n,3,a}=X^n(X^{3(q-1)}+a)$ is a PB of $\f_{q^2}$ if and only if $\text{\rm gcd}(n,q-1)=1$, $n\equiv 3\pmod{q+1}$ and $a^{q+1}\ne 1$. 
\end{res}

\noindent
(Note: In the original statement of Result~\ref{R1.5} in \cite{Tu-Zeng-Jiang-Li-DCC-2021}, it is assumed that $m\ge 2$. However, the result also holds for $m=1$; see Example~\ref{Tu-Zeng}.)

\begin{res}[{\cite[Theorem~4.2]{Hou-arXiv:1609.03662}}]\label{R1.6}
Assume $e=2$ and $d=1$. Then $f_{q,2,n,1,a}=X^n(X^{q-1}+a)$ is a PB of $\f_{q^2}$ if and only if $\text{\rm gcd}(n,q-1)=1$, $n\equiv 1\pmod{q+1}$ and $a^{q+1}\ne 1$. 
\end{res}

\begin{res}[\cite{Masuda-Rubio-Santiago-arXiv:2009.10851}]\label{R1.7}
Assume $e\ge 2$, $d=1$ and $n<q^e-q$. For the special cases $(q,e)=(q,2), (q,3), (q,4), (p,5), (p,6)$, where $p$ is a prime, the following statement is true: If $f_{q,e,n,1,a}=X^n(X^{q-1}+a)$ is a PB of $\f_{q^e}$, then $f_{q,e,n,1,a}\equiv X^{nq^h}+aX^n\pmod{X^{q^e}-X}$ for some integer $h>0$. It is conjectured that the statement is true for all $q$.
\end{res}

\noindent
(Note: In Result~\ref{R1.7}, when $q=2$, $f_{2,e,n,1,a}=X^n(X+a)$ is never a PB of $\f_{2^e}$, so the statement is vacuously true.)

\medskip

Through these results, we begin to understand the roles played by the parameters in the PBs of the form \eqref{f}. At the same time, as more results on PBs gather, one feels a need for a properly defined notion of {\em equivalence} of PBs that allows us to categorize existing results and channel future efforts to PBs that are new under equivalence. Section~2 is included for this purpose. We define the equivalence among all PBs (not just those of the form \eqref{f}). We show that every PB can be brought to a canonical form which is uniquely determined by a triple of invariants. In particular, we see that the PB in Result~\ref{R1.4} is equivalent to a PB in Result~\ref{R1.2} and the PB in Result~\ref{R1.5} is equivalent to a PB in Result~\ref{R1.6}.

Regarding Result~\ref{R1.5}, if we assume $e=2$, $q=2^{2m+1}$, $d=3$ and $a^{q+1}\ne 1$, \cite{Tu-Zeng-Jiang-Li-DCC-2021} conjectured that $f_{q,2,n,3,a}=X^n(X^{3(q-1)}+a)$ is not a PB of $\f_{q^2}$ and provided strong evidence for this conjecture. Note that in this case, $d\mid q+1$. As we will see in Section~2, when the PB in \eqref{f} is brought to its canonical form, we always have $d\mid (q^e-1)/(q-1)$.

Let us further focus on the case $e=2$, and we assume $d\mid q+1$ by the above comment. In this case, if $a^{q+1}=1$ or $d=1$, all PBs are given by Results~\ref{R1.1} and \ref{R1.6}. Therefore, we assume $e=2$, $2\le d\mid q+1$ and $a^{q+1}\ne 1$. Under these assumptions and up to equivalence, the PBs in Result~\ref{R1.2} form the only known class that contains infinitely many $q$'s. This leads to the following question.

\begin{ques}\label{Q1.8}
Fix integers $n\ge 1$ and $d\ge 2$. If there are infinitely many pairs $(q,a)$ such that $d\mid q+1$, $a\in\f_{q^2}^*$, $a^{q+1}\ne 1$, and $f(X)=f_{q,2,n,d,a}(X)=X^n(X^{d(q-1)}+a)$ is a PB of $\f_{q^2}$, can we conclude that when $q$ is sufficiently large, $f$ is equivalent the PB in Result~\ref{R1.2}?
\end{ques}

In this paper, we prove two nonexistence results that support an affirmative answer to the above question.

\begin{thm}\label{T1.9}
Let $q=2^m$, $n\ge 1$ and $a\in\f_{q^2}^*$ be such that $q\ge (2\max\{n,6-n\})^4$ and $a^{q+1}\ne 1$. Then $f(X)=f_{q,2,n,3,a}(X)=X^n(X^{3(q-1)}+a)$ is not a PB of $\f_{q^2}$.
\end{thm}

Theorem~\ref{T1.9} proves the conjecture of \cite{Tu-Zeng-Jiang-Li-DCC-2021} when $q$ is large relative to $n$.

\begin{thm}\label{T1.10}
Let $n\ge 1$, $d\ge 2$ and $a\in\f_{q^2}^*$ be such that $d\mid q+1$, $q\ge (2\max\{n,2d-n\})^4$ and $a^{q+1}\ne 1$. Then $f(X)=f_{q,2,n,d,a}(X)=X^n(X^{d(q-1)}+a)$ is not a PB of $\f_{q^2}$ if one of the following conditions is satisfied.
\begin{itemize}
\item[(i)] $d-n>1$ and $\text{\rm gcd}(d,n+1)$ is a power of $2$.

\item[(ii)] $d+2\le n<2d$ and $\text{\rm gcd}(d,n-1)$ is a power of $2$. 

\item[(iii)] $n\ge 2d$, $\text{\rm gcd}(d,n-1)$ is a power of $2$, and $\text{\rm gcd}(n-d,q-1)=1$.
\end{itemize}
\end{thm}

\begin{rmk}\label{Rmk1.11}\rm
In Theorem~\ref{T1.10}, one can replace the assumption that $d\mid q+1$ with $\text{gcd}(n,d)=1$. If the $f$ in Theorem~\ref{T1.10} is a PB of $\f_{q^2}$, then $d\mid q+1$ implies $\text{gcd}(n,d)=1$. However, as we will see in Section~4, the proof of Theorem~\ref{T1.10} only uses $\text{gcd}(n,d)=1$. Moreover, the assumption that $\text{gcd}(n,d)=1$ causes no loss of generality. If $f_{q,2,n,d,a}$ is a PB of $\f_{q^2}$ with $\text{gcd}(n,d)=\delta$, then $\text{gcd}(\delta,q^2-1)=1$. Let $\delta'\in\Bbb Z^+$ be such that $\delta\delta'\equiv 1\pmod{q^2-1}$. Then
\[
f_{q,2,n,d,a}(X^{\delta'})\equiv f_{q,2,n/\delta,d/\delta,a}(X)\pmod{X^{q^2}-X},
\]
where $\text{gcd}(n/\delta,d/\delta)=1$.
\end{rmk}

Result~\ref{R1.3} is a special case of Theorem~\ref{T1.10} (i) with $n=1$. Although the conditions in Theorem~\ref{T1.10} are rather restrictive, they do cover many parameters that were not investigated previously. For example, (i) is satisfied for all $d>n+1$ with $\text{gcd}(d,n+1)=1$.

Theorems~\ref{T1.9} and \ref{T1.10} are proved in Sections~3 and 4, respectively. The method is similar to that in \cite{Hou-FFA-2018}. Here we recall the basic strategy.

Let
\begin{equation}\label{fq2}
f(X)=f_{q,2,n,d,a}(X)=X^n(X^{d(q-1)}+a)\in\f_{q^2}[X],
\end{equation}
where $n\ge 1$, $2\le d\mid q+1$ and $a\in\f_{q^2}^*$. The following theorem follows from a well-known folklore \cite{Park-Lee-BAMS-2001, Wang-LNCS-2007, Zieve-PAMS-2009}.

\begin{thm}\label{folklore}
The binomial $f(X)$ in \eqref{fq2} is PB of $\f_{q^2}$ if and only if 
\begin{itemize}
\item[(i)] $\text{\rm gcd}(n,d(q-1))=1$ and
\item[(ii)] $X^n(X^d+a)^{q-1}$ permutes $\mu_{q+1}:=\{x\in\f_{q^2}^*: x^{q+1}=1\}$.
\end{itemize}
\end{thm}

Assume that $f(X)$ in \eqref{fq2} is a PB of $\f_{q^2}$. Then for $x\in\mu_{q+1}$,
\[
x^n(x^d+a)^{q-1}=\frac{x^n(x^{dq}+a^q)}{x^d+a}=\frac{x^n(a^qx^d+1)}{x^d(x^d+a)}=G(x),
\]
where
\begin{equation}\label{G}
G(X)=\frac{a^qX^n+X^{n-d}}{X^d+a}.
\end{equation}
Write
\[
G(X)=\frac{P(X)}{Q(X)},
\]
where
\[
\begin{cases}
P(X)=a^qX^n+X^{n-d},\cr
Q(X)=X^d+a,
\end{cases}
\qquad\text{if}\ n\ge d,
\]
\[
\begin{cases}
P(X)=a^qX^d+1,\cr
Q(X)=X^{2d-n}+aX^{d-n},
\end{cases}
\qquad\text{if}\ n<d.
\]
We assume that $a^{q+1}\ne 1$, which implies that $\text{gcd}(P,Q)=1$. Thus
\[
\deg G=\begin{cases}
n&\text{if}\ n\ge d,\cr
2d-n&\text{if}\ n<d.
\end{cases}
\]
Let
\begin{equation}\label{N(G)}
N(G)=\frac{P(X)Q(Y)-P(Y)Q(X)}{X-Y}\in\f_{q^2}[X,Y],
\end{equation}
which is the numerator of $(G(X)-G(Y))/(X-Y)$.
We have
\[
\deg N(G)\le\begin{cases}
n+d-1&\text{if}\ n\ge d,\cr
3d-n-1&\text{if}\ n<d.
\end{cases}
\]

\begin{thm}\label{HW}
Assume that $f(X)$ in \eqref{fq2} is a PB of $\f_{q^2}$, where $q\ge (2\max\{n,2d-n\})^4$. Then $N(G)$ in \eqref{N(G)} is reducible in $\overline\f_q[X,Y]$, where $\overline\f_q$ is the algebraic closure of $\f_q$.
\end{thm}

\begin{proof}
We only give a sketch of the proof; the omitted details are given in \cite[\S 3]{Hou-FFA-2018}.

There exist $l_1,l_2\in\f_{q^2}(X)$ of degree one such that $H:=l_1\circ G\circ l_2$ permutes $\f_q$. Since $\deg H=\deg G<q$, by \cite[Lemma~3.2]{Hou-FFA-2018}, $H\in\f_q(X)$. Let $A(X,Y)=N(H)\in\f_q[X,Y]$, the numerator of $(H(X)-H(Y))/(X-Y)$. Assume to the contrary that $N(G)$ is irreducible over $\overline\f_q$. Then by \cite[Lemma~3.1]{Hou-FFA-2018}, $A(X,Y)$ is also irreducible over $\overline\f_q$. We have
\[
\delta:=\deg A(X,Y)\le 2\deg H-2=2\deg G-2.
\]
By the Hasse-Weil bound, the number of zeros of $A(X,Y)$ in the projective plane $\Bbb P^2(\f_q)$ is at least
\[
q-(\delta-1)(\delta-2)q^{1/2}.
\]
Excluding the zeros at infinity of $\Bbb P^2(\f_q)$ and on the diagonal $\{(x,x):x\in\f_q\}$ of the affine plane $\f_q^2$, we have
\[
|\{(x,y)\in\f_q^2:x\ne y,\ A(x,y)=0\}|\ge q-(\delta-1)(\delta-2)q^{1/2}-2\delta.
\]
The right side is positive since $q\ge \delta^4$. Hence there exist $(x,y)\in\f_q^2$ with $x\ne y$ such that $A(x,y)=0$. Then $H(x)=H(y)$, which is a contradiction.
\end{proof}

\section{Canonical Forms of Permutation Binomials}

For our purpose, a binomial over $\f_q$ is a polynomial of the 
\[
f(X)=aX^m+bX^n\in\f_q[X],
\]
where $a,b\in\f_q^*$, $m,n>0$, $m\not\equiv 0$, $n\not\equiv 0$ and $m\not\equiv n\pmod{q-1}$. We treat $f(X)$ as a function from $\f_q$ to $\f_q$, that is, we identify $f(X)$ with its image in the quotient ring $\f_q[X]/\langle X^q-X\rangle$. Let $\mathcal B_q$ denote the set of all such binomials. Two members $f,g\in\mathcal B_q$ are considered {\em equivalent}, denoted as $f\sim g$, if one can be obtained from the other through a combination of the following transformations of $\mathcal B_q$: 
\begin{align}\label{alpha}
&\alpha_u:\mathcal B_q\to\mathcal B_q,\ f(X)\mapsto uf(X),\quad u\in\f_q^*,\\
\label{beta}
&\beta:\mathcal B_q\to\mathcal B_q,\ f(X)\mapsto f(X)^p,\quad p=\text{char}\,\f_q,\\
\label{gamma}
&\gamma_{v,s}:\mathcal B_q\to\mathcal B_q,\ f(X)\mapsto f(vX^s),\quad v\in\f_q^*,\ s\in\Bbb Z^+,\ \text{gcd}(s,q-1)=1.
\end{align}
If $f,g\in\mathcal B_q$ are equivalent, then $f$ permutes $\f_q$ if and only if $g$ does. It is clear that $\gamma_{v,s}$ commutes with $\alpha_u$ and $\beta$, and $\beta\circ\alpha_u=\alpha_{u^p}\circ\beta$. Therefore, for $f,g\in\mathcal B_q$, $f\sim g$ if and only if
\begin{equation}\label{g=uf}
g(X)=uf(vX^s)^{p^i}
\end{equation}
for some $u,v\in\f_q^*$, $i\ge 0$ and $s>0$ with $\text{gcd}(s,q-1)=1$.

For $d\mid q-1$, define
\begin{equation}\label{def-Nd}
N_d=\{1\le n\le q-1:n=n^*\},
\end{equation}
where
\begin{align*}
n^*=\min\bigl\{1\le n'\le q-1:\ &\text{$n'\equiv tn\!\!\pmod{q-1}$ for some $t\in\Bbb Z_{q-1}^\times$}\cr
&\text{with $t\equiv 1\!\!\pmod{(q-1)/d}$ or}\cr
&\text{$n'\equiv tn-d\!\!\pmod{q-1}$ for some $t\in\Bbb Z_{q-1}^\times$}\cr
&\text{with $t\equiv -1\!\!\pmod{(q-1)/d}$}\bigr\}. 
\end{align*}
(Here $\Bbb Z_{q-1}^\times$ denotes the multiplicative group of $\Bbb Z_{q-1}$.) Let $\theta:\Bbb Z_{q-1}^\times\to\Bbb Z_{(q-1)/d}^\times$ be the natural homomorphism (which is onto). Then $G:=\theta^{-1}(\{\pm1\})$ acts on $\Bbb Z_{q-1}$ as follows: For $t\in G$ and $n\in\Bbb Z_{q-1}$,
\[
t(n)=\begin{cases}
tn&\text{if}\ \theta(t)=1,\cr
tn-d&\text{if}\ \theta(t)=-1.
\end{cases}
\]
Write $\Bbb Z_{q-1}=\{1,2,\dots,q-1\}$. Then for $n\in\Bbb Z_{q-1}$, $n^*$ is the least element in the $G$-orbit of $n$. Therefore $N_d$ is the set of least elements of the $G$-orbits in $\Bbb Z_{q-1}$.

\begin{exmp}\label{Exmp-Nd}\rm
Let $q=2^4$ and $d=3$. We have $\theta:\Bbb Z_{15}^\times\to\Bbb Z_5^\times$, $\theta^{-1}(1)=\{1,11\}$, $\theta^{-1}(-1)=\{-1,4\}$ and $G=\{1,11,-1,4\}$. The $G$-orbits of $\Bbb Z_{15}$ are $\{1,11\}$, $\{2,7,10,5\}$, $\{3,9\}$, $\{4,14,8,13\}$, $\{6\}$, $\{15\}$. Hence $N_d=\{1,2,3,4,6,15\}$.
\end{exmp}

For $d\mid q-1$ and $n\in N_d$, let
\begin{align}\label{Gd,n}
G_{d,n}=\,&\text{the subgroup of $\Bbb Z_d^\times$ generated by}\\
&\begin{cases}
\{p,-1\}&\text{if $d\equiv -2n\pmod{(q-1)/d}$ and $\text{gcd}(n,q-1)=1$},\cr
\{p\}&\text{otherwise},
\end{cases}\nonumber
\end{align} 
where $p=\text{char}\,\f_q$. Let $G_{d,n}$ act on $\f_q^*/(\f_q^*)^d$, where $(\f_q^*)^d=\{x^d:x\in\f_q^*\}$, as follows:
\[
\begin{array}{ccl}
G_{d,n}\times\; \f_q^*/(\f_q^*)^d & \longrightarrow & \f_q^*/(\f_q^*)^d \vspace{0.5em}\\
(s, a(\f_q^*)^d) & \longmapsto & a^s(\f_q^*)^d,\quad a\in\f_q^*.
\end{array}
\]
Let $A_{d,n}\subset\f_q^*$ be such that $\{a(\f_q^*)^d:a\in A_{d,n}\}$ is a system of representatives of the $G_{d,n}$-orbits in $\f_q^*/(\f_q^*)^d$. Equivalently, let $G_{d,n}$ act on $\Bbb Z_d$ through multiplication and let $\xi$ be a primitive element of $\f_q$. Then $A_{d,n}=\{\xi^e:e\in E_{d,n}\}$, where $E_{d,n}$ is a system of representatives of the $G_{d,n}$-orbits in $\Bbb Z_d$.

We now are ready to state and prove the main result of this section.

\begin{thm}\label{T-can-form}
Assume that $f\in\mathcal B_q$ permutes $\f_q$. Then there is a unique triple $(d,n,a)$, where $d\mid q-1$, $n\in N_d$ and $a\in A_{d,n}$, such that 
\begin{equation}\label{can-form}
f(X)\sim X^n(X^d+a).
\end{equation}
We call the right side of \eqref{can-form} the {\em canonical form} of $f$.
\end{thm}

\begin{proof}
{\em Existence of $(d,n,a)$}

Write $f(X)=a_0X^{m_0}+b_0X^{n_0}$, where $a_0,b_0\in\f_q^*$ and $m_0>n_0$. Let $d=\text{gcd}(m_0-n_0,q-1)$. Let $r\in\Bbb Z^+$ be such that 
\[
r\frac{m_0-n_0}d\equiv 1\pmod{\frac{q-1}d}.
\]
Since $\text{gcd}(r,(q-1)/d)=1$, there exists integer $k\ge 0$ such that $s:=r+k(q-1)/d$ is relatively prime to $q-1$. (To see this, use Dirichlet's theorem on primes in arithmetic progression or the following simple argument: Let $p_1,\dots,p_l$ be the prime divisors of $q-1$ that do not divide $r$ and let $k=p_1\cdots p_l$.) Then
\[
f(X)\sim f(X^s)=X^{sn_0}(a_0X^{s(m_0-n_0)}+b_0)=X^{n_1}(a_0X^d+b_0),
\]
where $n_1=sn_0$. We now assume $f(X)=X^{n_1}(a_0X^d+b_0)$.

Let $n=n_1^*\in N_d$. We claim that
\begin{equation}\label{claim}
f(X)\sim X^n(a_1X^d+b_1)
\end{equation}
for some $a_1,b_1\in\f_q^*$. To prove this claim, we consider two cases.

{\bf Case 1.}
Assume that $n\equiv tn_1\pmod{q-1}$ for some $t\in\Bbb Z_{q-1}^\times$ with $t\equiv 1\pmod{(q-1)/d}$. Then
\[
f(X)\sim f(X^t)=X^{tn_1}(a_0X^{td}+b_0)=X^n(a_0X^d+b_0).
\]

{\bf Case 2.} 
Assume that $n\equiv tn_1-d\pmod{q-1}$ for some $t\in\Bbb Z_{q-1}^\times$ with $t\equiv -1\pmod{(q-1)/d}$. Then 
\begin{align*}
f(X)\sim f(X^t)\,&=X^{tn_1}(a_0X^{td}+b_0)=a_0X^{tn_1+td}+b_0X^{tn_1}\cr
&=a_0X^n+b_0X^{n+d}=X^n(b_0X^d+a_0).
\end{align*}
Hence \eqref{claim} is proved.

By \eqref{claim}, we may assume
\[
f(X)=X^n(X^d+c),
\]
where $c\in\f_q^*$. To prove that $f(X)\sim X^n(X^d+a)$ for some $a\in A_{d,n}$, again, we consider two cases.

{\bf Case 1.} Assume that $d\not\equiv -2n\pmod{(q-1)/d}$ or $\text{gcd}(n,q-1)\ne 1$. By \eqref{Gd,n}, $G_{d,n}=\langle p\rangle<\Bbb Z_d^\times$. Then by the definition of $A_{d,n}$, there exist $i\in\Bbb N$, $a\in A_{d,n}$ and $b\in\f_q^*$ such that $c^{p^i}=ab^d$. Write $b=b_1^{p^i}$, where $b_1\in\f_q^*$. Let $s\in\Bbb Z^+$ be such that $sp^i\equiv 1\pmod{q-1}$. Then
\begin{align*}
f(X)\,&\sim f(b_1 X^s)^{p^i}=(b_1 X^s)^{np^i}((b_1 X^s)^{dp^i}+c^{p^i})\cr
&\sim X^n(b_1^{dp^i}X^d+c^{p^i})=X^n(b^dX^d+c^{p^i})\cr
&\sim X^n(X^d+c^{p^i}b^{-d})=X^n(X^d+a).
\end{align*}

{\bf Case 2.} Assume that $d\equiv -2n\pmod{(q-1)/d}$ and $\text{gcd}(n,q-1)=1$. Then $G_{d,n}=\langle p,-1\rangle<\Bbb Z_d^\times$. So there exist $i\in\Bbb N$, $a\in A_{d,n}$ and $b\in\f_q^*$ such that either $c^{p^i}=ab^d$ or $c^{-p^i}=ab^d$. In the former case, the proof is identical to Case 1. In the latter case, write $b=b_1^{p^i}$, where $b_1\in\f_q^*$. Let $k\in\Bbb Z^+$ be such that $kn\equiv 1\pmod{q-1}$, and let $s=1+kd$. Then 
\begin{align*}
sn=n+nkd&\,\equiv n+d\pmod{q-1}\cr
&\equiv -n\pmod{(q-1)/d}.
\end{align*}
Hence $s\equiv -1\pmod{(q-1)/d}$. It follows that $\text{gcd}(s,(q-1)/d)=1$. We also have $\text{gcd}(s,d)=\text{gcd}(1+kd,d)=1$. Therefore $\text{gcd}(s,q-1)=1$. We have
\[
f(X)\sim f(X^s)=X^{sn}(X^{sd}+c)=X^{sn+sd}+cX^{sn}.
\]
In the above,
\[
sn=n+nkd\equiv n+d\pmod{q-1}
\]
and 
\[
sd=(1+kd)d\equiv(1+k(-2n))d\equiv -d\pmod{q-1}.
\]
Hence
\[
f(X)\sim X^n+cX^{n+d}\sim X^n(X^d+c^{-1}),
\]
where $(c^{-1})^{p^i}=ab^d$. It follows from Case 1 that
\[
X^n(X^d+c^{-1})\sim X^n(X^d+a).
\]

\medskip
{\em Uniqueness of $(d,n,a)$}

Assume that
\begin{equation}\label{x^n_x^n1}
f(X)=X^n(X^d+a)\sim X^{n_1}(X^{d_1}+a_1),
\end{equation}
where $d\mid q-1$, $n\in N_d$, $a\in A_{d,n}$, $d_1\mid q-1$, $n_1\in N_{d_1}$, $a_1\in A_{d_1,n_1}$.

In general, for $bX^m+cX^l\in\mathcal B_q$, $\text{gcd}(m-l,q-1)$ is invariant under equivalence. Therefore, in \eqref{x^n_x^n1}, we have $d=d_1$.

By \eqref{x^n_x^n1},
\begin{equation}\label{x^n1=ufv}
X^{n_1}(X^d+a_1)=uf(vX^s)^{p^i}
\end{equation}
for some $u,v\in\f_q^*$, $i\ge 0$ and $s>0$ with $\text{gcd}(s,q-1)=1$. Expanding \eqref{x^n1=ufv} gives
\[
X^{n_1+d}+a_1X^{n_1}=\alpha X^{t(n+d)}+\beta X^{tn},
\]
where $t=sp^i$ and $\alpha,\beta\in\f_q^*$. It follows that 
\begin{equation}\label{n1=d=t(n+d)}
\begin{cases}
n_1+d\equiv t(n+d)\pmod{q-1},\cr
n_1\equiv tn\pmod{q-1},
\end{cases}
\end{equation}
or
\begin{equation}\label{n1=d=tn}
\begin{cases}
n_1+d\equiv tn\pmod{q-1},\cr
n_1\equiv t(n+d)\pmod{q-1}.
\end{cases}
\end{equation}
Note that \eqref{n1=d=t(n+d)} is equivalent to
\begin{equation}\label{t=1}
\begin{cases}
t\equiv 1\pmod{(q-1)/d},\cr
n_1\equiv tn\pmod{q-1},
\end{cases}
\end{equation}
and \eqref{n1=d=tn} is equivalent to
\begin{equation}\label{t=-1}
\begin{cases}
t\equiv -1\pmod{(q-1)/d},\cr
n_1\equiv tn-d\pmod{q-1}.
\end{cases}
\end{equation}
Since $n\in N_d$, it follows from \eqref{t=1}, \eqref{t=-1} and the definition of $N_d$ (\eqref{def-Nd}) that $n\le n_1$. By symmetry, $n_1\le n$, whence $n=n_1$.

Now \eqref{x^n1=ufv} becomes
\begin{align*}
X^{n+d}+a_1X^n\,&=u\bigl[(vX^s)^{n+d}+a(vX^s)^n\bigr]^{p^i}\cr
&=uv^{p^i(n+d)}X^{sp^i(n+d)}+ua^{p^i}v^{p^in}X^{sp^in}.
\end{align*}
Let $t=sp^i$. Then there are two possibilities.

{\bf Case 1.} \eqref{t=1} holds with $n=n_1$ and
\begin{equation}\label{1a1case1}
(uv^{p^i(n+d)},\, ua^{p^i}v^{p^in})=(1, a_1).
\end{equation}

{\bf Case 2.} \eqref{t=-1} holds with $n=n_1$ and
\begin{equation}\label{1a1case2}
(ua^{p^i}v^{p^in},\, uv^{p^i(n+d)})=(1, a_1).
\end{equation}

\noindent
It suffices to show that in both cases, $a$ and $a_1$ are in the same $G_{d,n}$-orbit. (Then $a=a_1$.)

First, assume Case 1. We have 
\[
a_1=\frac{ua^{p^i}v^{p^in}}{uv^{p^i(n+d)}}=a^{p^i}v^{-p^id},
\]
which is in the $G_{d,n}$-orbit of $a$.

Next, assume Case 2. \eqref{t=-1} with $n=n_1$ gives
\[
\begin{cases}
t\equiv -1\pmod{(q-1)/d},\cr
n\equiv tn-d\pmod{q-1}.
\end{cases}
\]
It follows that $n\equiv tn-d\equiv -n-d\pmod{(q-1)/d}$, i.e., $d\equiv -2n\pmod{(q-1)/d}$. Since $f(X)$ permutes $\f_q$, we have $\text{gcd}(n,d)=1$. From $n\equiv tn-d\pmod{q-1}$, we have $(t-1)n-d\equiv 0\pmod{q-1}$, whence $d\mid t-1$ and 
\[
\frac{t-1}dn-1\equiv 0\pmod{\frac{q-1}d}.
\]
In particular, $\text{gcd}(n,(q-1)/d)=1$. Combining this with $\text{gcd}(n,d)=1$, we have $\text{gcd}(n,q-1)=1$. Therefore $G_{d,n}=\langle p,-1\rangle$. Now by \eqref{1a1case2},
\[
a_1=\frac{uv^{p^i(n+d)}}{ua^{p^i}v^{p^in}}=a^{-p^i}v^{p^id},
\]
which is in the $G_{d,n}$-orbit of $a$.
\end{proof}

\begin{exmp}\label{Exmp-Lappano}\rm
Assume that $n,d\in\Bbb Z^+$ are such that $d\mid q+1$, $n<2d$, $\text{gcd}(n,q^2-1)=1$ and $\text{gcd}(2d-n,q-1)=1$, and let $a\in\f_{q^2}^*$. Since $\text{gcd}(dq-n+d,q-1)=\text{gcd}(2d-n,q-1)=1$ and $\text{gcd}(dq-n+d,q+1)=\text{gcd}(n,q+1)=1$, we have $\text{gcd}(dq-n+d,q^2-1)=1$. Then in $\mathcal B_q$,
\begin{align*}
X^n(X^{d(q-1)}+a)\,&=X^{dq+n-d}+aX^n\cr
&\sim X^{(dq-n+d)(dq+n-d)}+aX^{(dq-n+d)n}\kern2em (X\mapsto X^{dq-n+d})\cr
&=X^{d^2q^2-(n-d)^2}+aX^{(dq-n+d)n}\cr
&=X^{d^2-(n-d)^2}+aX^{(dq-n+d)n}\cr
&=X^{n(2d-n)}+aX^{(dq-n+d)n}\cr
&\sim X^{2d-n}+aX^{dq-n+d} \kern 8.3em (X^n\mapsto X)\cr
&=X^{2d-n}(1+aX^{d(q-1)})\cr
&\sim X^{2d-n}(X^{d(q-1)}+a^{-1}).
\end{align*}
In particular, when $n=1$, $d=2$, $q$ is odd and $q\not\equiv 1\pmod 3$, we have 
\[
X(X^{2(q-1)}+a)\sim X^3(X^{2(q-1)}+a^{-1}).
\]
This shows that the PB in Result~\ref{R1.4} is equivalent to a PB in Result~\ref{R1.2}. 
\end{exmp}

\begin{exmp}\label{Tu-Zeng}\rm
We show that the PB in Result~\ref{R1.5} is equivalent to a PB in Result~\ref{R1.6}. Let $e=2$, $q=2^{2m}$, $n\in\Bbb Z^+$, $d=3$, $a\in\f_{q^2}^*$, and consider $f=f_{q,2,n,3,a}=X^n(X^{3(q-1)}+a)$. 

Let $s=(q+2)/3+k(q+1)$, where 
\[
k=\begin{cases}
0&\text{if}\ m\equiv 0,1\pmod 3,\cr
1&\text{if}\ m\equiv -1\pmod 3.
\end{cases}
\]
We claim that $\text{gcd}(s,q^2-1)=1$. Clearly, $\text{gcd}(s,q+1)=1$. We have 
\begin{align*}
\text{gcd}(s,q-1)\,&=\text{gcd}\Bigl(\frac{q+2}3+2k,\,q-1\Bigr)\cr
&=\frac 13\text{gcd}(q+2+6k,\,3q-3)\cr
&=\frac 13\text{gcd}(q+2+6k,\,3(-2-6k)-3)\cr
&=\frac 13\text{gcd}(q+2+6k,\,9(2k+1)).
\end{align*}
In the above, $9(2k+1)=3^2$ or $3^3$, and 
\begin{align*}
q+2+6k\,&=(3-1)^{2m}+2+6k\cr
&\equiv 1-2m\cdot 3+2+6k\pmod{3^2}\cr
&=3+6(k-m)\cr
&\not\equiv 0\pmod{3^2}.
\end{align*}
So $\text{gcd}(s,q-1)=1$ and the claim is proved.

Now we have
\[
f(X)\sim f(X^s)=X^{sn}(X^{s\cdot 3(q-1)}+a)=X^{sn}(X^{q-1}+a).
\]
By Result~\ref{R1.6}, $X^{sn}(X^{q-1}+a)$ permutes $\f_{q^2}$ if and only if
\[
\text{gcd}(sn,q-1)=1,\quad sn\equiv 1\pmod{q+1},\quad\text{and}\ a^{q+1}\ne 1,
\]
i.e.,
\[
\text{gcd}(n,q-1)=1,\quad n\equiv 3\pmod{q+1},\quad\text{and}\ a^{q+1}\ne 1,
\]
which are precisely the conditions in Result~\ref{R1.5}.
\end{exmp}

\section{Proof of Theorem~\ref{T1.9}}

\noindent{\bf Theorem~\ref{T1.9}.} {\it
Let $q=2^m$, $n\ge 1$ and $a\in\f_{q^2}^*$ be such that $q\ge (2\max\{n,6-n\})^4$ and $a^{q+1}\ne 1$. Then $f(X)=f_{q,2,n,3,a}(X)=X^n(X^{3(q-1)}+a)$ is not a PB of $\f_{q^2}$.}

\medskip

Assume to the contrary that $f$ is a PB of $\f_{q^2}$. If $m$ is even, by Result~\ref{R1.5}, $n\ge q+4$, which is a contradiction. So $m$ is odd, and $3\mid q+1$.
By \eqref{G}, 
\begin{equation}\label{G(X)}
G(X)=\frac{a^qX^n+X^{n-3}}{X^3+a}.
\end{equation}
Let 
\begin{equation}\label{N(X,Y)}
N(X,Y)=\text{the numerator of}\ \frac{G(X)+G(Y)}{X+Y}.
\end{equation}
By Theorem~\ref{HW}, $N(X,Y)$ is reducible over $\overline\f_q$. However, we will show that $N(X,Y)$ is irreducible over $\overline\f_q$, hence creating a contradiction. We consider two cases, $n\ge 3$ and $n\le 2$, separately. 

\subsection{Case 1. $n\ge 3$}\

Since $\text{gcd}(n,3(q-1))=1$ (Theorem~\ref{folklore}), we have $n>3$. We have
\begin{align*}
N(X,Y)=\,&\frac 1{X+Y}\Bigl[(a^qX^n+X^{n-3})(Y^3+a)+(a^qY^n+Y^{n-3})(X^3+a)\Bigr]\cr
=\,&a\frac{X^{n-3}+Y^{n-3}}{X+Y}+\Bigl[a^{q+1}\frac{X^n+Y^n}{X+Y}+X^3Y^3\frac{X^{n-6}+Y^{n-6}}{X+Y}\Bigr]\cr
&+a^qX^3Y^3\frac{X^{n-3}+Y^{n-3}}{X+Y}.
\end{align*}
The homogenization of $N(X,Y)$ is 
\begin{align*}
N^*(X,Y,Z)=\,&a\frac{X^{n-3}+Y^{n-3}}{X+Y}Z^6+\Bigl[a^{q+1}\frac{X^n+Y^n}{X+Y}+X^3Y^3\frac{X^{n-6}+Y^{n-6}}{X+Y}\Bigr]Z^3\cr
&+a^qX^3Y^3\frac{X^{n-3}+Y^{n-3}}{X+Y}\cr
=\,&Q(Z^3),
\end{align*}
where
\begin{align*}
Q(Z)=\,& a\frac{X^{n-3}+Y^{n-3}}{X+Y}Z^2+\Bigl[a^{q+1}\frac{X^n+Y^n}{X+Y}+X^3Y^3\frac{X^{n-6}+Y^{n-6}}{X+Y}\Bigr]Z\cr
&+a^qX^3Y^3\frac{X^{n-3}+Y^{n-3}}{X+Y}.
\end{align*}
It suffices to show that $N^*(X,Y,Z)$ is irreducible over $\overline\f_q$. We first show that $N^*(X,Y,Z)$, as a polynomial in $Z$ over $\overline\f_q[X,Y]$, is primitive, i.e., the gcd of its coefficients is 1; that is,
\begin{equation}\label{gcd}
\text{gcd}\Bigl(\frac{X^{n-3}+Y^{n-3}}{X+Y},\, a^{q+1}\frac{X^n+Y^n}{X+Y}+X^3Y^3\frac{X^{n-6}+Y^{n-6}}{X+Y}\Bigr)=1.
\end{equation}  
Since the polynomials in \eqref{gcd} are homogeneous, it suffices to prove \eqref{gcd} with $Y=1$, i.e.,  
\begin{equation}\label{gcdY=1}
\text{gcd}\Bigl(\frac{X^{n-3}+1}{X+1},\, a^{q+1}\frac{X^n+1}{X+1}+X^3\frac{X^{n-6}+1}{X+1}\Bigr)=1.
\end{equation}  
Let $\zeta\in\overline\f_q$ be a root of $(X^{n-3}+1)/(X+1)$. If $\zeta\ne 1$, then $\zeta^{n-3}+1=0$. Thus
\begin{align*}
&\Bigl(a^{q+1}\frac{X^n+1}{X+1}+X^3\frac{X^{n-6}+1}{X+1}\Bigr)\Big|_{X=1}\cr
=\,&
\frac1{\zeta+1}\bigl(a^{q+1}(\zeta^n+1)+\zeta^3(\zeta^{n-6}+1)\bigr)\cr
=\,&\frac1{\zeta+1}\bigl(a^{q+1}(\zeta^3+1)+1+\zeta^3\bigr)\cr
=\,&\frac1{\zeta+1}(a^{q+1}+1)(\zeta^3+1)\ne 0.
\end{align*}
(Note: $\zeta^3\ne 1$ since $\zeta^{n-3}=1$ and $\text{gcd}(n,3(q-1))=1$.)
If $\zeta=1$, then $n$ must be odd, in which case, 
\[
\Bigl(a^{q+1}\frac{X^n+1}{X+1}+X^3\frac{X^{n-6}+1}{X+1}\Bigr)\Big|_{X=1}=a^{q+1}n+n-6=n(a^{q+1}+1)\ne 0.
\]
This proves \eqref{gcdY=1} and hence \eqref{gcd}.

With \eqref{gcd}, to prove that $N^*(X,Y,Z)$ is irreducible in $\overline\f_q[X,Y,Z]$, it suffices to show that it is irreducible in $\overline\f_q(X,Y)[Z]$.
Let $w$ be a root of $N^*(X,Y,Z)\in\overline\f_q(X,Y)[Z]$ and let $z=w^3$. Then $z$ is a root of $Q(Z)$. It suffices to show that $[\overline\f_q(X,Y,z):\overline\f_q(X,Y)]=2$ and $[\overline\f_q(X,Y,w):\overline\f_q(X,Y,z)]=3$.

\[
\beginpicture
\setcoordinatesystem units <4mm,4mm> point at 0 0

\setlinear
\plot 0 1  0 3 /
\plot 0 5  0 7 /

\put {$\overline{\f}_q(X,Y)$} at 0 0
\put {$\overline{\f}_q(X,Y,z)$} at 0 4
\put {$\overline{\f}_q(X,Y,w)$}  at 0 8
\put {$\scriptstyle 2$} [l] at 0.2 2
\put {$\scriptstyle 3$} [l] at 0.2 6
\endpicture
\]


\subsubsection{Proof that $[\overline\f_q(X,Y,z):\overline\f_q(X,Y)]=2$}\

Assume to the contrary that $Q(Z)$ is reducible over $\overline\f_q(X,Y)$. Then there exists $A/B\in\overline\f_q(X,Y)$ ($A,B\in\overline\f_q[X,Y]$, $\text{gcd}(A,B)=1$) such that
\begin{equation}\label{A(A+B)}
\frac{\displaystyle a^{q+1}X^3Y^3\Bigl(\frac{X^{n-3}+Y^{n-3}}{X+Y}\Bigr)^2}{\displaystyle \Bigl(a^{q+1}\frac{X^n+Y^n}{X+Y}+X^3Y^3\frac{X^{n-6}+Y^{n-6}}{X+Y}\Bigr)^2}=
\Bigl(\frac AB\Bigr)^2+\frac AB=\frac{A(A+B)}{B^2}.
\end{equation}
In the above equation, the numerator and the denominator on the left side are relatively prime (by \eqref{gcd}), so
\begin{equation}\label{B}
B=a^{q+1}\frac{X^n+Y^n}{X+Y}+X^3Y^3\frac{X^{n-6}+Y^{n-6}}{X+Y}
\end{equation}
and
\[
A(A+B)=a^{q+1}X^3Y^3\Bigl(\frac{X^{n-3}+Y^{n-3}}{X+Y}\Bigr)^2.
\]
Since $\text{gcd}(A,A+B)=1$, we may assume that 
\begin{equation}\label{A=X3U2}
\begin{cases}
A=X^3U^2,\cr
A+B=Y^3V^2,
\end{cases}
\end{equation}
for some $U,V\in\overline\f_q[X,Y]$ with $UV=(X^{n-3}+Y^{n-3})/(X+Y)$. Therefore,
\begin{equation}\label{B=}
B=X^3U^2+Y^3V^2.
\end{equation}
By \eqref{B=}, the coefficient of $XY^{n-2}$ in $B$ is $0$. However, by \eqref{B}, the coefficient of $XY^{n-2}$ in $B$ is either $a^{q+1}$ or $a^{q+1}+1$. We have a contradiction.

\subsubsection{Proof that $[\overline\f_q(X,Y,w):\overline\f_q(X,Y,z)]=3$}\

Assume the contrary. Then $z$ is a third power in $\overline\f_q(X,Y,z)$, that is, there exists $A,B\in\overline\f_q(X,Y)$ such that
\[
z=(A+Bz)^3,
\]
i.e.,
\begin{equation}\label{A+BZ}
(A+BZ)^3-Z\equiv 0\pmod{Q(Z)}.
\end{equation}
Setting $Y=1$ in \eqref{A+BZ} gives
\begin{equation}\label{A1+B1Z} 
(A_1+B_1Z)^3-Z\equiv 0\pmod{Q_1(Z)},
\end{equation}
where $A_1(X)=A(X,1)$, $B_1(X)=B(X,1)$ and   
\begin{align}\label{s3-Q1}
&Q_1(Z)=Q(Z)|_{Y=1}=\\ 
&a\frac{X^{n-3}+1}{X+1}Z^2+\Bigl[a^{q+1}\frac{X^n+1}{X+1}+X^3\frac{X^{n-6}+1}{X+1}\Bigr]Z+a^qX^3\frac{X^{n-3}+1}{X+1}.
\nonumber
\end{align}
We find that
\[
(A_1+B_1Z)^3-Z\equiv \frac{f_0(X)}{a^2(X^3+X^n)}+\frac{f_1(X)}{a^2(X^3+X^n)^2}Z\pmod{Q_1(Z)},
\]
where
\begin{align}\label{f0}
f_0(X)=\,& a^2 A_1^3 X^3 + a^{1+q} A_1 B_1^2 X^6 + a^{1+2 q} B_1^3 X^6 + a^q B_1^3 X^9 + a^2 A_1^3 X^n \\
&+ a^{1+q} A_1 B_1^2 X^{3+n} + a^q B_1^3 X^{3+n} + a^{1+2 q} B_1^3 X^{6+n},\cr
f_1(X)=\,& a^2 X^6 + a^2 A_1^2 B_1 X^6 + a^{2+q} A_1 B_1^2 X^6 + a^{2+2 q} B_1^3 X^6 + a A_1 B_1^2 X^9 \cr
&+ a^{1+q} B_1^3 X^9 +
B_1^3 X^{12} + a^2 X^{2 n} + a^2 A_1^2 B_1 X^{2 n} + a A_1 B_1^2 X^{2 n} + B_1^3 X^{2 n}\cr
& + a A_1 B_1^2 X^{3+n} + a^{2+q} A_1 B_1^2 X^{3+n} +
a A_1 B_1^2 X^{6+n} + a^{2+q} A_1 B_1^2 X^{6+n}\cr
& + a^{2+q} A_1 B_1^2 X^{3+2 n} + a^{1+q} B_1^3 X^{3+2 n} + a^{2+2 q} B_1^3 X^{6+2 n}.\nonumber
\end{align}
Therefore, $f_0(X)=f_1(X)=0$. (We will not need the fact that $f_1(X)=0$.) From \eqref{A1+B1Z}, $B_1\ne 0$. Then $f_0(X)=0$ implies $A_1\ne 0$. Let $C=B_1/A_1$. Then $f_0(X)=0$ becomes
\begin{equation}\label{a^2+1}
(a^2+a^{1+q}X^3C^2)(1+X^{n-3})=a^qX^3(a^{1+q}+X^3+X^{n-3}+a^{1+q}X^n)C^3.
\end{equation}
In the above
\begin{align*}
&\text{gcd}(1+X^{n-3},\, a^{1+q}+X^3+X^{n-3}+a^{1+q}X^n)\cr
=\,&\text{gcd}(1+X^{n-3},\, a^{1+q}+X^3+1+a^{1+q}X^3)\cr
=\,&\text{gcd}(1+X^{n-3},\, (a^{1+q}+1)(1+X^3))\cr
=\,&1+X.
\end{align*}
Let $C=D/E$, where $D,E\in\overline\f_q[X]$, $E$ is monic and $\text{gcd}(D,E)=1$. Then \eqref{a^2+1} becomes
\begin{equation}\label{a^2E}
(a^2E^3+a^{1+q}X^3D^2E)\frac{1+X^{n-3}}{1+X}=a^qX^3D^3\,\frac{a^{1+q}+X^3+X^{n-3}+a^{1+q}X^n}{1+X}.
\end{equation}
It follows that
\begin{equation}\label{D}
\frac{1+X^{n-3}}{1+X}\Bigm | D\quad\text{and}\quad D\Bigm | \frac{1+X^{n-3}}{1+X}.
\end{equation}
\eqref{a^2E} and \eqref{D} force $D\in\overline\f_q^*$ and $n=4$. So 
\[
a^2E^3+a^{1+q}D^2X^3E=a^qD^3X^3(a^{1+q}(1+X)^3+X(1+X)).
\]
Then $X\mid E$, say $E=XE_1$. Thus
\begin{equation}\label{a^2E1}
a^2E_1^3+a^{1+q}D^2XE_1=a^qD^3(1+X)(a^{1+q}X^2+X+a^{1+q}).
\end{equation}
It follows that $\deg E_1=1$, say $E_1=X+\epsilon$, $\epsilon\in\overline\f_q$. Comparing the coefficients of $X^3$ and $X^0$ in the above gives 
\begin{align}\label{a^2}
a^2\,&=a^{1+2q}D^3, \\
a^2\epsilon^3\,&=a^{1+2q}D^3.\nonumber
\end{align}
Hence $\epsilon^3=1$. Then $(a^{1+q}X^2+X+a^{1+q})|_{X=\epsilon}=a^{1+q}(1+\epsilon^2)+\epsilon\ne 0$ since $a^{1+q}\ne 1$.  It follows from \eqref{a^2E1} that $E_1\mid 1+X$, that is, $E_1=X+1$. Now \eqref{a^2E1} becomes
\[
a^2(X+1)^2+a^{1+q}D^2X=a^qD^3(a^{1+q}X^2+X+a^{1+q}).
\]
Comparing the coefficients of $X$ in the above gives $a^{1+q}D^2=a^qD^3$, i.e., $D=a$. But then \eqref{a^2} gives $a^{1+q}=1$, which is a contradiction.

\subsection{Case 2. $n\le 2$}\ 

When $n=1$, the absolute irreducibility (irreducibility over $\overline\f_q$) of $N(X,Y)$ follows from \cite[\S 3]{Hou-FFA-2018}. So we assume $n=2$. The arguments are similar to those in Case 1. We have
\begin{equation}\label{case1-G}
G(X)=\frac{a^qX^3+1}{X(X^3+a)},
\end{equation}
\begin{equation}\label{case2-N}
N(X,Y)=a^qX^3Y^3+a^{q+1}XY(X+Y)+(X+Y)^3+a,
\end{equation}
and 
\begin{equation}\label{case2-Q}
Q(Z)=aZ^2+(a^{q+1}XY(X+Y)+(X+Y)^3)Z+a^qX^3Y^3.
\end{equation}
When proving $[\overline\f_q(X,Y,z):\overline\f_q(X,Y)]=2$, equations \eqref{A(A+B)}, \eqref{B} and \eqref{A=X3U2} are replaced by
\[
\frac{a^{q+1}X^3Y^3}{(a^{q+1}XY(X+Y)+(X+Y)^3)^2}=\frac{A(A+B)}{B^2},
\]
\begin{equation}\label{case2-B}
B=a^{q+1}XY(X+Y)+(X+Y)^3,
\end{equation}
and 
\[
\begin{cases}
A=uX^3,\cr
A+B=vY^3,
\end{cases}
\quad u,v\in\overline\f_q^*.
\]
Then $B=uX^3+vY^3$, which contradicts \eqref{case2-B} since $a^{1+q}\ne 1$.

When proving $[\overline\f_q(X,Y,w):\overline\f_q(X,Y,z)]=3$, equation \eqref{f0} is replaced by
\begin{align}\label{case2-f0}
f_0(X)=\,&a^2 A_1^3 + a^{1+q} A_1 B_1^2 X^3 + a^q B_1^3 X^3 +  a^q B_1^3 X^4 + a^{1+2 q} B_1^3 X^4 +  a^q B_1^3 X^5 \\
&+ a^{1+2 q} B_1^3 X^5 + a^q B_1^3 X^6.\nonumber
\end{align}
Setting $E=A_1/B_1$, the equation $f_0(X)=0$ becomes
\[
a^2 E^3+a^{1+q} X^3E+a^q X^3(1+X)(1+a^{1+q}X+X^2)=0.
\]
It follows that $E\in\overline\f_q[X]$ and $X\mid E$. Write $E=XE_1$. Then
\begin{equation}\label{case2-a2E13}
a^2 E_1^3+a^{1+q} X E_1+a^q (1 + X)(1+a^{1+q}X+X^2)=0.
\end{equation}
Thus $\deg E_1=1$, say $E_1=e(X+\epsilon)$, $e\in\overline\f_q^*$, $\epsilon\in\overline\f_q$. Comparing the coefficients of $X^3$ and $X^0$ in the above gives
\begin{align}\label{case2-a2e3} 
&a^2e^3+a^q=0,\\
&a^2e^3\epsilon^3+a^q=0. \nonumber
\end{align}
Hence $\epsilon^3=1$. Then $(1+a^{1+q}X+X^2)|_{X=\epsilon}\ne 0$. It follows form \eqref{case2-a2E13} that $E_1\mid 1+X$, whence $E_1=e(X+1)$. Then \eqref{case2-a2E13} becomes 
\[
a^2e^3(X+1)^2+a^{1+q}eX+a^q(1+a^{1+q}X+X^2)=0.
\]
Comparing the coefficients of $X$ in the above gives $e=a^q$. But then \eqref{case2-a2e3} gives $a^{1+q}=1$, which is a contradiction.

\begin{rmk}\rm
Most likely, Theorem~\ref{T1.9} also holds for odd $q$.
\end{rmk}

\section{Proof of Theorem~\ref{T1.10}}

\noindent{\bf Theorem~\ref{T1.10}.}
{\it
Let $n\ge 1$, $d\ge 2$ and $a\in\f_{q^2}^*$ be such that $d\mid q+1$, $q\ge (2\max\{n,2d-n\})^4$ and $a^{q+1}\ne 1$. Then $f(X)=f_{q,2,n,d,a}(X)=X^n(X^{d(q-1)}+a)$ is not a PB of $\f_{q^2}$ if one of the following conditions is satisfied.
\begin{itemize}
\item[(i)] $d-n>1$ and $\text{\rm gcd}(d,n+1)$ is a power of $2$.

\item[(ii)] $d+2\le n<2d$ and $\text{\rm gcd}(d,n-1)$ is a power of $2$. 

\item[(iii)] $n\ge 2d$, $\text{\rm gcd}(d,n-1)$ is a power of $2$, and $\text{\rm gcd}(n-d,q-1)=1$.
\end{itemize}
}

\medskip

Assume to the contrary that $f(X)$ is a PB of $\f_{q^2}$. Recall that
\[
G(X)=\frac{a^qX^n+X^{n-d}}{X^d+a}.
\]
Let
\[
N(X,Y)=\text{the numerator of}\ \frac{G(X)-G(Y)}{X-Y}
\]
and
\[
N^*(X,Y,Z)=\text{the homogenization of}\ N(X,Y).
\]
Our objective is to show that $N^*(X,Y,Z)$ is irreducible over $\overline\f_q$ under the conditions in Theorem~\ref{T1.10}. We consider two cases: the case $d-n>1$, which corresponds to (i) in Theorem~\ref{T1.10}, and the case $n-d>1$, which corresponds to (ii) and (iii) in Theorem~\ref{T1.10}.

\subsection{The case $d-n>1$}\label{s4.1}\

We have
\[
G(X)=\frac{a^qX^d+1}{X^{d-n}(X^d+a)},
\]
\begin{align*}
N(X,Y)=\,&-a\frac{X^{d-n}-Y^{d-n}}{X-Y}+\Bigl[a^{q+1}X^{d-n}Y^{d-n}\frac{X^n-Y^n}{X-Y}-\frac{X^{2d-n}-Y^{2d-n}}{X-Y}\Bigr]\cr
&-a^qX^dY^d\frac{X^{d-n}-Y^{d-n}}{X-Y},
\end{align*}
\[
N^*(X,Y,Z)=Q(Z^d),
\]
where
\begin{align}\label{QZ}
Q(Z)=\,&-a\frac{X^{d-n}-Y^{d-n}}{X-Y}Z^2+\Bigl[a^{q+1}X^{d-n}Y^{d-n}\frac{X^n-Y^n}{X-Y}-\frac{X^{2d-n}-Y^{2d-n}}{X-Y}\Bigr]Z\\
&-a^qX^dY^d\frac{X^{d-n}-Y^{d-n}}{X-Y}. \nonumber
\end{align}
We claim that
\begin{equation}\label{gcd=1}
\text{gcd}\Bigl(\frac{X^{d-n}-Y^{d-n}}{X-Y},\,a^{q+1}X^{d-n}Y^{d-n}\frac{X^n-Y^n}{X-Y}-\frac{X^{2d-n}-Y^{2d-n}}{X-Y}\Bigr)=1.
\end{equation}
Since the polynomials in \eqref{gcd=1} are homogeneous, it suffices to prove \eqref{gcd=1} with $Y=1$, i.e.,
\begin{equation}\label{gcd=1Y=1}
\text{gcd}\Bigl(\frac{X^{d-n}-1}{X-1},\,a^{q+1}X^{d-n}\frac{X^n-1}{X-1}-\frac{X^{2d-n}-1}{X-1}\Bigr)=1.
\end{equation}
Let $\zeta$ be a root of $(X^{d-n}-1)/(X-1)$. If $\zeta\ne 1$, then $\zeta^{d-n}=1$. It follows that 
\begin{align*}
&\Bigl(a^{q+1}X^{d-n}\frac{X^n-1}{X-1}-\frac{X^{2d-n}-1}{X-1}\Bigr)\Big|_{X=\zeta}\cr
&=\frac 1{\zeta-1}\bigl(a^{q+1}(\zeta^n-1)-(\zeta^n-1)\bigr)=\frac 1{\zeta-1}(a^{q+1}-1)(\zeta^n-1)\ne0.
\end{align*}
(Note: $\zeta^n\ne 1$ since $\zeta^{d-n}=1$ and $\text{gcd}(n,d)=1$.) If $\zeta=1$, then $d-n\equiv 0\pmod p$, where $p=\text{char}\,\f_q$, whence
\[
\Bigl(a^{q+1}X^{d-n}\frac{X^n-1}{X-1}-\frac{X^{2d-n}-1}{X-1}\Bigr)\Big|_{X=1}=a^{q+1}n-(2d-n)=(a^{q+1}-1)n\ne 0.
\]
This proves \eqref{gcd=1Y=1} and hence \eqref{gcd=1}. By \eqref{gcd=1}, $N^*(X,Y,Z)$ is a primitive polynomial in $Z$ over $\overline\f_q[X,Y]$, i.e., the gcd of its coefficients in $\overline\f_q[X,Y]$ is 1. Thus, to prove that $N^*(X,Y,Z)$ is irreducible in $\overline\f_q[X,Y,Z]$, it suffices to show that it is irreducible in $\overline\f_q(X,Y)[Z]$. Let $w$ be a root of $N^*(X,Y,Z)$ for $Z$ and let $z=w^d$. Then $z$ is a root of $Q(Z)$, and it suffices to show that $[\overline\f_q(X,Y,z):\overline\f_q(X,Y)]=2$ and $[\overline\f_q(X,Y,w):\overline\f_q(X,Y,z)]=d$.

\[
\beginpicture
\setcoordinatesystem units <4mm,4mm> point at 0 0

\setlinear
\plot 0 1  0 3 /
\plot 0 5  0 7 /

\put {$\overline{\f}_q(X,Y)$} at 0 0
\put {$\overline{\f}_q(X,Y,z)$} at 0 4
\put {$\overline{\f}_q(X,Y,w)$}  at 0 8
\put {$\scriptstyle 2$} [l] at 0.2 2
\put {$\scriptstyle d$} [l] at 0.2 6
\endpicture
\]


\subsubsection{Proof that $[\overline\f_q(X,Y,z):\overline\f_q(X,Y)]=2$}\

Assume to the contrary that $Q(Z)$ is reducible over $\overline\f_q(X,Y)$.

{\em First assume that $q$ is odd}. The discriminant of $Q$ is 
\[
D=\Bigl[a^{q+1}X^{d-n}Y^{d-n}\frac{X^n-Y^n}{X-Y}-\frac{X^{2d-n}-Y^{2d-n}}{X-Y} \Bigr]^2-4a^{q+1}X^dY^d\Bigl(\frac{X^{d-n}-Y^{d-n}}{X-Y} \Bigr)^2.
\]
By assumption, $D=\Delta^2$ for some $\Delta\in\overline\f_q[X,Y]$. Then
\begin{align}\label{4aq+1}
&4a^{q+1}X^dY^d\Bigl(\frac{X^{d-n}-Y^{d-n}}{X-Y} \Bigr)^2=\\
&\Bigl[a^{q+1}X^{d-n}Y^{d-n}\frac{X^n-Y^n}{X-Y}-\frac{X^{2d-n}-Y^{2d-n}}{X-Y}+\Delta \Bigr]\cr
&\cdot \Bigl[a^{q+1}X^{d-n}Y^{d-n}\frac{X^n-Y^n}{X-Y}-\frac{X^{2d-n}-Y^{2d-n}}{X-Y}-\Delta \Bigr]. \nonumber
\end{align}
Let $\delta$ be the gcd of the two factors on the right side of \eqref{4aq+1}. Then 
\[
\delta\Bigm | a^{q+1}X^{d-n}Y^{d-n}\frac{X^n-Y^n}{X-Y}-\frac{X^{2d-n}-Y^{2d-n}}{X-Y}
\]
and 
\[
\delta\bigm | \frac{X^{d-n}-Y^{d-n}}{X-Y}.
\]
It follows from \eqref{gcd=1} that $\delta=1$.

Now from \eqref{4aq+1}, we have
\[
\begin{cases}
\displaystyle a^{q+1}X^{d-n}Y^{d-n}\frac{X^n-Y^n}{X-Y}-\frac{X^{2d-n}-Y^{2d-n}}{X-Y}+\Delta=X^dU,\vspace{0.4em}\cr
\displaystyle a^{q+1}X^{d-n}Y^{d-n}\frac{X^n-Y^n}{X-Y}-\frac{X^{2d-n}-Y^{2d-n}}{X-Y}-\Delta=Y^dV,
\end{cases}
\]
for some $U,V\in\overline\f_q[X,Y]$. It follows that
\begin{equation}\label{xdU+ydV}
2a^{q+1}X^{d-n}Y^{d-n}\frac{X^n-Y^n}{X-Y}-2\frac{X^{2d-n}-Y^{2d-n}}{X-Y}=X^dU+Y^dV.
\end{equation}
The coefficient of $X^{d-1}Y^{d-n}$ on the left side of \eqref{xdU+ydV} is $2(a^{q+1}-1)\ne 0$, while the coefficient of the same term on the right side of \eqref{xdU+ydV} is 0. This is a contradiction.

\medskip
{\em Next, assume that $q$ is even}. Since $Q(Z)$ is assumed to be reducible over $\overline\f_q(X,Y)$, we have
\[
\frac{\displaystyle a^{q+1}X^dY^d\Bigl(\frac{X^{d-n}+Y^{d-n}}{X+Y}\Bigr)^2}{\displaystyle \Bigl[a^{q+1}X^{d-n}Y^{d-n}\frac{X^n+Y^n}{X+Y}+\frac{X^{2d-n}+Y^{2d-n}}{X+Y}\Bigr]^2}=\Bigl(\frac AB\Bigr)^2+\frac AB=\frac{A(A+B)}{B^2},
\]
where $A,B\in\overline\f_q[X,Y]$, $\text{gcd}(A,B)=1$. By \eqref{gcd=1}, the numerator and the denominator on the left side are relatively prime. Therefore we may assume
\begin{equation}\label{B=aq+1}
B=a^{q+1}X^{d-n}Y^{d-n}\frac{X^n+Y^n}{X+Y}+\frac{X^{2d-n}+Y^{2d-n}}{X+Y},
\end{equation}
\[
A(A+B)=a^{q+1}X^dY^d\Bigl(\frac{X^{d-n}+Y^{d-n}}{X+Y}\Bigr)^2.
\]
Since $\text{gcd}(A,A+B)=1$, we have
\[
\begin{cases}
A=X^dU^2,\cr
A+B=Y^dV^2,
\end{cases}
\]
where $U,V\in\overline\f_q[X,Y]$, $UV=(X^{d-n}+Y^{d-n})/(X+Y)$. Then
\begin{equation}\label{B-qeven}
B=X^dU^2+Y^dV^2.
\end{equation}
The coefficient of $X^{d-1}Y^{d-n}$ in \eqref{B=aq+1} is $a^{q+1}+1\ne 0$. However, the coefficient of $X^{d-1}Y^{d-n}$ in \eqref{B-qeven} is 0, which is a contradiction.

\subsubsection{Proof that $[\overline\f_q(X,Y,w):\overline\f_q(X,Y,z)]=d$}\

To prove this claim, it suffices to show that for each prime divisor $t$ of $d$, $z$ is not a $t$th power in $\overline\f_q(X,Y,z)$. In \eqref{QZ}, divide $Q(Z)$ by its leading coefficient and set $Y=1$, the result is 
\begin{equation}\label{Q1}
Q_1(Z)=Z^2-\frac{\displaystyle a^{q+1}X^{d-n}\frac{X^n-1}{X-1}-\frac{X^{2d-n}-1}{X-1}}{\displaystyle a\frac{X^{d-n}-1}{X-1}}Z+a^{q-1}X^d,
\end{equation}
which is irreducible in $\overline\f_q(X)[Z]$. Let $z_1$ be a root of $Q_1(Z)$. By \cite[\S 3.3, Claim II$'$]{Hou-FFA-2018}, it suffices to show that for each prime divisor $t$ of $d$, $z_1$ is not a $t$th power in $\overline\f_q(X,z_1)$.

Let $\overline{(\ )}$ denote the nonidentity automorphism in $\text{Aut}(\overline\f_q(X,z_1)/\overline\f_q(X))$. We have
\begin{equation}\label{zbarz}
z_1\bar z_1=a^{q-1}X^d,
\end{equation}
\begin{equation}\label{z+barz}
z_1+\bar z_1=\frac{\displaystyle a^{q+1}X^{d-n}\frac{X^n-1}{X-1}-\frac{X^{2d-n}-1}{X-1}}{\displaystyle a\frac{X^{d-n}-1}{X-1}}.
\end{equation}
Let $d-n=p^md'$, where $p=\text{char}\,\f_q$, $p\nmid d'$. Let $\zeta\in\overline\f_q$ be a primitive $d'$th root of unity. Let $\frak p$ be the place of the rational function field $\overline\f_q(X)$ which is the zero of $X-\zeta$, and let $\frak P$ be a place of $\overline\f_q(X,z_1)$ such that $\frak P\mid \frak p$. Then $\frak P$ is unramified over $\frak p$ (\cite[III 7.3 (b) and 7.8 (b)]{Stichtenoth-1993}). From \eqref{zbarz} and \eqref{z+barz}, we have
\begin{equation}\label{vz}
\nu_{\frak p}(z_1\bar z_1)=0,
\end{equation}
\begin{equation}\label{vz+}
\nu_{\frak p}(z_1+\bar z_1)=
\begin{cases}
-p^m&\text{if}\ d'>1,\cr
-p^m+1&\text{if}\ d'=1,
\end{cases}
\end{equation}
where $\nu_{\frak p}$ is the valuation of $\overline\f_q(X)$ at $\frak p$. Equation~\eqref{vz+} is derived as follows: First, note that in \eqref{z+barz},
\begin{equation}\label{vp1}
\nu_{\frak p}\Bigl(\frac{X^{d-n}-1}{X-1} \Bigr)=
\begin{cases}
p^m&\text{if}\ d'>1,\cr
p^m-1&\text{if}\ d'=1.
\end{cases}
\end{equation}
Next, write
\begin{equation}\label{write}
a^{q+1}X^{d-n}\frac{X^n-1}{X-1}-\frac{X^{2d-n}-1}{X-1}=(a^{q+1}X^{d-n}-1)\frac{X^n-1}{X-1}-X^n\frac{X^{2(d-n)}-1}{X-1}.
\end{equation}
If $d'>1$, the value of \eqref{write} at $X=\zeta$ is 
\[
(a^{q+1}-1)\frac{\zeta^n-1}{\zeta-1}\ne 0.
\]
If $d'=1$, we have $m>0$ (since $d-n>1$), whence $d-n\equiv 0\pmod p$. Then $n\not\equiv 0$ since $\text{gcd}(n,d)=1$. Therefore, the value of \eqref{write} at $X=\zeta$ ($=1$) is 
\[
(a^{q+1}-1)n-2(d-n)=(a^{q+1}-1)n\ne 0.
\] 
Hence we always have
\begin{equation}\label{vp2}
\nu_{\frak p}\Bigl(a^{q+1}X^{d-n}\frac{X^n-1}{X-1}-\frac{X^{2d-n}-1}{X-1} \Bigr)=0.
\end{equation}
Combining \eqref{z+barz}, \eqref{vp1} and \eqref{vp2} gives \eqref{vz+}.

Write \eqref{vz} and \eqref{vz+} as
\[
\nu_{\frak P}(z_1)+\nu_{\frak P}(\bar z_1)=0,
\]
\[
\nu_{\frak P}(z_1+\bar z_1)=
\begin{cases}
-p^m&\text{if}\ d'>1,\cr
-p^m+1&\text{if}\ d'=1,
\end{cases}
\]
where $\nu_{\frak P}$ is the valuation of $\overline\f_q(X,z_1)$ at $\frak P$. It follows that
\begin{equation}\label{vP}
\{\nu_{\frak P}(z_1),\nu_{\frak P}(\bar z_1)\}=\begin{cases}
\{\pm p^m\}&\text{if}\ d'>1,\cr
\{\pm(p^m-1)\}&\text{if}\ d'=1.
\end{cases}
\end{equation}
Assume to the contrary that $z_1$ is a $t$th power  in $\overline\f_q(X,z_1)$. Then $t\mid\nu_{\frak P}(z_1)$. If $d'>1$, the by \eqref{vP}, $t\mid p^m$, whence $t\mid d-n$. This is impossible since $t\mid d$ and $\text{gcd}(n,d)=1$. Therefore, $d'=1$ and $d-n=p^m$. By \eqref{vP}, $t\mid p^m-1=d-n-1$. Since $t\mid\text{gcd}(d,d-n-1)=\text{gcd}(d,n+1)$ and $\text{gcd}(d,n+1)$ is a power of $2$, we have $t=2$. It follows that $p$ is odd.

Recall that $Q_1(z_1)=0$, where $Q_1(Z)$ is given in \eqref{Q1}. Using \eqref{write} and $d-n=p^m$, the equation $Q_1(z_1)=0$ can be written as
\[
u^2=\delta,
\]
where
\[
u=z_1-\gamma,
\]
\[
\gamma=\frac 12\frac{\displaystyle (a^{q+1}X^{p^m}-1)\frac{X^n-1}{X-1}-X^n(X+1)^{p^m}(X-1)^{p^m-1}}{a(X-1)^{p^m-1}},
\]
and
\[
\delta=\gamma^2-a^{q-1}X^{p^m+n}.
\]
By assumption, there exist $\alpha,\beta\in\overline{\Bbb F}_q(X)$ such that
\[
(\alpha u+\beta)^2=u+\gamma,
\] 
i.e.,
\[
\alpha^2\delta+\beta^2+2\alpha\beta u=u+\gamma.
\]
Since $u$ is of degree 2 over $\overline\f_q(X)$, we have
\[
\begin{cases}
\alpha^2\delta+\beta^2=\gamma,\cr
2\alpha\beta=1.
\end{cases}
\]
Letting $\tau=\alpha/\beta$, we have
\begin{equation}\label{eq-tau}
1+\delta\tau^2-2\gamma\tau=0
\end{equation}
and
\begin{equation}\label{tau}
\tau=2\alpha^2.
\end{equation}
Fortunately, \eqref{eq-tau} has an explicit solution
\[
\tau=\frac 1\delta(\gamma\pm a^{(q-1)/2}X^{(p^m+n)/2})=\frac 1{\gamma\mp a^{(q-1)/2}X^{(p^m+n)/2}}.
\]
In the above,
\begin{align*}
&\gamma\mp a^{(q-1)/2}X^{(p^m+n)/2}=\cr
&\frac 1{2a(X-1)^{p^m-1}}\Bigl[(a^{q+1}X^{p^m}-1)\frac{X^n-1}{X-1}-X^n(X+1)^{p^m}(X-1)^{p^m-1}\cr
&\kern 7.4em \mp 2a^{(q+1)/2}X^{(p^m+n)/2}(X-1)^{p^m-1}\Bigr].
\end{align*}
Since $\tau$ is square in $\overline\f_q(X)$ (by \eqref{tau}), 
\[
h:= (1-a^{q+1}X^{p^m})\frac{X^n-1}{X-1}+X^n(X+1)^{p^m}(X-1)^{p^m-1}+ 2\epsilon X^{(p^m+n)/2}(X-1)^{p^m-1},
\]
where $\epsilon=\pm a^{(q+1)/2}$, is a square in $\overline\f_q(X)$, say $h=g^2$, where $g\in\overline\f_q[X]$ is monic of degree $p^m+(n-1)/2$. Note that
\begin{align*}
h=\,&\frac{X^n-1}{X-1}+(X^n+X^{p^m+n})\frac{X^{p^m}-1}{X-1}-a^{q+1}X^{p^m}\frac{X^n-1}{X-1}+ 2\epsilon X^{(p^m+n)/2}\frac{X^{p^m}-1}{X-1}\cr
=\,&(1+\cdots+X^{2p^m+n-1})\cr
&-a^{q+1}(X^{p^m}+\cdots+X^{p^m+n-1})\cr
&+ 2\epsilon (X^{(p^m+n)/2}+\cdots+X^{(3p^m+n)/2-1}),
\end{align*}
which is self-reciprocal. Hence $g^*=\pm g$, where $g^*$ is the reciprocal polynomial of $g$. (In fact, If $g^*=g$, but we do not need to be precise.) Let
\begin{align*}
H\,&=(X-1)h\cr
&=(1-a^{q+1}X^{p^m})(X^n-1)+X^n(X+1)^{p^m}(X-1)^{p^m}+ 2\epsilon X^{(p^m+n)/2}(X-1)^{p^m}.
\end{align*}
Then
\[
H'=(1-a^{q+1}X^{p^m})nX^{n-1}+nX^{n-1}(X+1)^{p^m}(X-1)^{p^m}+ \epsilon n X^{(p^m+n)/2-1}(X-1)^{p^m}.
\]
(When computing $H'$, we used the assumption that $m>0$.)
Let
\begin{align*}
K&=H-n^{-1}XH'=-(1-a^{q+1}X^{p^m})+\epsilon X^{(p^m+n)/2}(X-1)^{p^m}\cr
&=-1+a^{q+1}X^{p^m}-\epsilon X^{(p^m+n)/2}+\epsilon X^{(p^m+n)/2+p^m}.
\end{align*}
The reciprocal of $K$ is
\[
K^*=\epsilon-\epsilon X^{p^m}+a^{q+1}X^{(p^m+n)/2}-X^{(p^m+n)/2+p^m}.
\]
Since $g\mid K$ and $g$ is self-reciprocal, we also have $g=\pm g^*\mid K^*$. Thus $g$ divides
\[
K+\epsilon K^*=-1+\epsilon^2+(-\epsilon+\epsilon a^{q+1})X^{(p^m+n)/2}=(a^{q+1}-1)(1+\epsilon X^{(p^m+n)/2}).
\]
This is a contradiction since 
\[
\frac{p^m+n}2<p^m+\frac{n-1}2=\deg g.
\]

\subsection{The Case $n-d>1$}\

In this case,
\[
G(X)=\frac{a^qX^n+X^{n-d}}{X^d+a},
\]
\begin{align*}
N(X,Y)=\,&a\frac{X^{n-d}-Y^{n-d}}{X-Y}+\Bigl[a^{q+1}\frac{X^n-Y^n}{X-Y}+X^dY^d\frac{X^{n-2d}-Y^{n-2d}}{X-Y}\Bigr]\cr
&+a^qX^dY^d\frac{X^{n-d}-Y^{n-d}}{X-Y},
\end{align*}
\[
N^*(X,Y,Z)=Q(Z^d),
\]
where 
\begin{align*}
Q(Z)=\,&a\frac{X^{n-d}-Y^{n-d}}{X-Y}Z^2+\Bigl[a^{q+1}\frac{X^n-Y^n}{X-Y}+X^dY^d\frac{X^{n-2d}-Y^{n-2d}}{X-Y}\Bigr]Z\cr
&+a^qX^dY^d\frac{X^{n-d}-Y^{n-d}}{X-Y}.
\end{align*}
We claim that
\begin{equation}\label{4.2-gcd}
\gcd\Big(\frac{X^{n-d}-Y^{n-d}}{X-Y},\, a^{q+1} \frac{X^n-Y^n}{X-Y}+X^d Y^d \frac{X^{n-2 d}-Y^{n-2 d}}{X-Y}\Big)=1,
\end{equation}
equivalently,
\begin{equation}\label{4.2-gcd1}
\gcd\Bigl(\frac{X^{n-d}-1}{X-1},\, a^{q+1} \frac{X^n-1}{X-1}+X^d \frac{X^{n-2 d}-1}{X-1}\Bigr)=1.
\end{equation}
Let $\zeta$ be a root of $(X^{n-d}-1)/(X-1)$. If $\zeta \ne 1$, then
\begin{align*}
\Bigl(a^{q+1} \frac{X^n-1}{X-1}+X^d \frac{X^{n-2 d}-1}{X-1}\Bigr)\big|_{X=\zeta}=\,&\frac1{\zeta-1}\bigl(a^{q+1} (\zeta^n-1)+\zeta^d (\zeta^{n-2 d}-1)\bigr)\cr
=\,&\frac1{\zeta-1} (a^{q+1}-1)(\zeta^n-1)\ne0.
\end{align*}
If $\zeta=1$, then $n-d\equiv 0\pmod p$, and
\[
\Big(a^{q+1} \frac{X^n-1}{X-1}+X^d \frac{X^{n-2 d}-1}{X-1}\Big)\Big|_{X=1}=a^{q+1}n+n-2d =n(a^{q+1}-1)\ne 0.
\] 
So \eqref{4.2-gcd1} and \eqref{4.2-gcd} hold. Therefore $Q(Z)$ is a primitive polynomial over $\overline\f_q[X,Y]$.

Let 
\begin{align}\label{4.2-Q1}
Q_1(Z)\,&=\Bigl[\Bigl(a\frac{X^{n-d}-Y^{n-d}}{X-Y} \Bigr)^{-1}Q(Z)\Bigr]\Big |_{Y=1}\\
&=Z^2+\frac{\displaystyle a^{q+1}\frac{X^n-1}{X-1}+X^d \frac{X^{n-2 d}-1}{X-1}}{\displaystyle a\frac{X^{n-d}-1}{X-1}} Z+a^{q-1}X^d\in\overline\f_q(X)[Z]. \nonumber
\end{align}
Following the arguments in Section~\ref{s4.1}, we only have to prove the following two claims:

\smallskip
\noindent{\bf Claim 1.} $Q(Z)$ is irreducible in $\overline\f_q(X,Y)[Z]$.

\smallskip
\noindent{\bf Claim 2.} Let $z$ be a root of $Q_1(Z)$ and $t$ be a prime divisor of $d$. Then $z$ is not a $t$th power in $\overline\f_q(X,z)$.

\subsubsection{Proof of Claim 1}\

Assume to the contrary that $Q(Z)$ is reducible in $\overline\f_q(X,Y)[Z]$.

{\em First, assume that $q$ is odd}. The discriminant of $Q(Z)$ is 
\[
D=\Bigl[\frac{a^{q+1}(X^n-Y^n)}{X-Y}+\frac{X^d Y^d (X^{n-2 d}-Y^{n-2 d})}{X-Y}\Bigr]^2-\frac{4 a^{q+1} X^d Y^d (X^{n-d}-Y^{n-d})^2}{(X-Y)^2}.
\]
By assumption, $D=\Delta^2$ for some $\Delta\in\overline\f_q[X,Y]$. Then
\begin{align*}
\frac{4 a^{q+1} X^d Y^d \left(X^{n-d}-Y^{n-d}\right)^2}{(X-Y)^2}=\,&\Big(a^{q+1}\frac{X^n-Y^n}{X-Y}+X^d Y^d\frac{X^{n-2 d}-Y^{n-2 d}}{X-Y}+\Delta\Big)\cr
&\cdot\Big(a^{q+1}\frac{X^n-Y^n}{X-Y}+X^d Y^d\frac{X^{n-2 d}-Y^{n-2 d}}{X-Y}-\Delta\Big).
\end{align*}	
In the above, the two factors on the right side are relatively prime. (This follows from \eqref{4.2-gcd}.) Therefore, we may assume 
\begin{equation}\label{system}
\begin{cases}
\displaystyle a^{q+1}\frac{X^n-Y^n}{X-Y}+X^d Y^d\frac{X^{n-2 d}-Y^{n-2 d}}{X-Y}+\Delta= 2a^{(q+1)/2}X^d U^2,\vspace{0.5em}\cr
\displaystyle a^{q+1}\frac{X^n-Y^n}{X-Y}+X^d Y^d\frac{X^{n-2 d}-Y^{n-2 d}}{X-Y}-\Delta=2a^{(q+1)/2}Y^d V^2,
\end{cases}
\end{equation}
for some $U,V \in \overline{\mathbb{F}}_q[X,Y]$ with 
\begin{equation}\label{eq:UV}
UV=\frac{X^{n-d}-Y^{n-d}}{X-Y}. 
\end{equation}
Then
\begin{equation}\label{XdU2+}
\frac{a^{q+1} \left(X^n-Y^n\right)}{X-Y}+\frac{X^d Y^d \left(X^{n-2 d}-Y^{n-2 d}\right)}{X-Y}=a^{(q+1)/2}(X^dU^2+Y^dV^2).
\end{equation}
Let $L$ denote the left side of \eqref{XdU2+}. We have
\begin{align*}
L=\,&a^{q+1}(Y^{n-1}+XY^{n-2}+\cdots+X^{n-1})\cr
&+\begin{cases}
X^dY^{n-d-1}+X^{d+1}Y^{n-d-2}+\cdots+X^{n-d-1}Y^d&\text{if}\ n\ge 2d,\vspace{0.2em}\cr
-X^{n-d}Y^{d-1}-X^{n-d+1}Y^{d-2}-\cdots-X^{d-1}Y^{n-d}&\text{if}\ d+2\le n< 2d.
\end{cases}
\end{align*}
If $d+2\le n<2d$, the coefficient of $X^{d-1}Y^{n-d}$ in $L$ is $a^{q+1}-1\ne 0$, while the the coefficient of $X^{d-1}Y^{n-d}$ on the right side of \eqref{XdU2+} is 0, which is a contradiction. Hence Theorem~\ref{T1.10} (iii) holds. In particular, $\text{gcd}(n-d,q-1)=1$.

Since 
\[
\Delta(Y,X)^2=D(Y,X)=D(X,Y)=\Delta(X,Y)^2.
\]
we have $\Delta(Y,X)=\pm\Delta(X,Y)$. If $\Delta(Y,X)=\Delta(X,Y)$, then by \eqref{system}, $X^dU(X,Y)^2=Y^dU(Y,X)^2$. Then $Y\mid U(X,Y)$, which is a contradiction to \eqref{eq:UV}. Hence $\Delta(Y,X)=-\Delta(X,Y)$, and by \eqref{system}, 
\begin{equation}\label{U2=V2}
U(Y,X)^2=V(X,Y)^2.
\end{equation}  
By \eqref{U2=V2} and \eqref{eq:UV}, we have
\begin{align}\label{U2}
U(X,Y)^2\,&=\alpha\prod_{i=1}^{(n-d-1)/2}(X-\epsilon_i Y)^2,\\ 
\label{V2}
V(X,Y)^2\,&=\alpha^{-1}\prod_{i=1}^{(n-d-1)/2}(X-\epsilon_i^{-1} Y)^2,
\end{align}
where $\alpha,\beta\in\overline\f_q$ and $\epsilon_i\in\overline\f_q^*$ are such that
\[
\frac{X^{n-d}-Y^{n-d}}{X-Y}=\prod_{i=1}^{(n-d-1)/2}\bigl[(X-\epsilon_i Y)(X-\epsilon_i^{-1} Y)\bigr].
\]
We have
\begin{align*}
\alpha\,&=U(1,0)^2  \kern6em \text{(by \eqref{U2})}\cr
&=V(0,1)^2 \kern6em \text{(by \eqref{U2=V2})}\cr
&=\alpha^{-1}\prod_{i=1}^{(n-d-1)/2}\epsilon_i^{-2}\qquad\text{(by \eqref{V2})}.
\end{align*}
It follows that
\begin{equation}\label{alpha2}
\alpha^2=\prod_{i=1}^{(n-d-1)/2}\epsilon_i^{-2}.
\end{equation}
On the other hand, comparing the coefficients of $X^{n-1}$ in \eqref{XdU2+} gives $a^{q+1}=a^{(q+1)/2}\cdot\alpha$, i.e., $\alpha=a^{(q+1)/2}$. Since the $\epsilon_i$'s are roots of $X^{n-d}-1$, we have
\[
a^{(q+1)(n-d)}=\alpha^{2(n-d)}=1\qquad\text{(by \eqref{alpha2})}.
\]
This, combined with $a^{(q+1)(q-1)}=1$ and $\text{gcd}(n-d,q-1)=1$, implies that $a^{q+1}=1$, which is a contradiction.

\medskip
{\em Next, assume that $q$ is even}. Since $Q(Z)$ is assumed to be reducible over $\overline\f_q(X,Y)$, there are $A,B \in \overline{\mathbb{F}}_q[X,Y]$, relatively prime, such that
\[
	\frac{\displaystyle a^{q+1} X^d Y^d \Bigl(\frac{X^{n-d}+Y^{n-d}}{X+Y}\Bigr)^2}{\displaystyle\Bigl(a^{q+1} \frac{X^n+Y^n}{X+Y}+X^d Y^d \frac{X^{n-2 d}+Y^{n-2 d}}{X+Y}\Bigr)^2}=\Big(\frac{A}{B}\Big)^2+\frac{A}{B}=\frac{A(A+B)}{B^2}.
\]
In the above, the numerator and the denominator on the left side are relatively prime (by \eqref{4.2-gcd}). Thus
\begin{equation}\label{B1}
B=a^{q+1} \frac{X^n+Y^n}{X+Y}+X^d Y^d \frac{X^{n-2 d}+Y^{n-2 d}}{X+Y}
\end{equation}
and 
\[
A(A+B)=a^{q+1} X^d Y^d \Bigl(\frac{X^{n-d}+Y^{n-d}}{X+Y}\Bigr)^2.
\]
We may assume that
\[
\begin{cases}
A=X^dU^2, \\
A+B=Y^dV^2,
\end{cases}
\]
for some $U,V \in \overline{\mathbb{F}}_q[X,Y]$ such that $UV=(X^{n-d}+Y^{n-d})/(X+Y)$. Then
\begin{equation}\label{B2}
B=X^d U^2+Y^dV^2.	
\end{equation}
By \eqref{B1},
\begin{align}\label{B3}
B=\,&a^{q+1}(Y^{n-1}+XY^{n-2}+\cdots+X^{n-1})\\
&+\begin{cases}
X^dY^{n-d-1}+X^{d+1}Y^{n-d-2}\cdots+X^{n-d-1}Y^d&\text{if}\ n\ge 2d,\cr
X^{n-d}Y^{d-1}+X^{n-d+1}Y^{d-2}+\cdots+X^{d-1}Y^{n-d}&\text{if}\ d+2\le n< 2d.
\end{cases}\nonumber
\end{align}
Since we assume $d>1$ and $n-d>1$, the coefficient of $XY^{n-2}$ in \eqref{B3} is $a^{q+1}\ne 0$. (Even if we allowed $d=1$ or $n-d=1$, the coefficient of $XY^{n-2}$ in \eqref{B3} would be $a^{q+1}+1$, which is still nonzero.) However, the coefficient of $XY^{n-2}$ in \eqref{B2} is 0, which is a contradiction.

\subsubsection{Proof of Claim 2}\

Recall that $Q_1(Z)$ is given in \eqref{4.2-Q1}. Let $z$ be a root of $Q_1(Z)$ and $t$ be a prime divisor of $d$. Assume to the contrary that $z$ is a $t$th power in $\overline\f_q(X,z)$. Let $\overline{(\ )}$ be the nonidentity automorphism in $\text{Aut}(\overline\f_q(X,z)/\overline\f_q(X))$. Then 
\begin{equation}\label{z1}
z\bar z=a^{q-1}X^d,
\end{equation}
\begin{align}\label{z2}
z+\bar z\,&=-\frac{\displaystyle a^{q+1}\frac{X^n-1}{X-1}+X^d\frac{X^{n-2d}-1}{X-1}}{\displaystyle a\frac{X^{n-d}-1}{X-1}}\cr
&=-\frac{\displaystyle (a^{q+1}-1)\frac{X^d-1}{X-1}+(a^{q+1}X^d+1)\frac{X^{n-d}-1}{X-1}}{\displaystyle a\frac{X^{n-d}-1}{X-1}}.
\end{align}
Write $n-d=p^md'$, where $p\nmid d'$, and let $\zeta$ be a primitive $d'$th root of unity. Let $\frak p$ be the place of the rational function field $\overline\f_q(X)$ which is the zero of $X-\zeta$, and let $\frak P$ be a place of $\overline\f_q(X,z)$ such that $\frak P\mid \frak p$. Then $\frak P$ is unramified over $\frak p$ (\cite[III 7.3 (b) and 7.8 (b)]{Stichtenoth-1993}). From \eqref{z1} and \eqref{z2}, we have
\begin{equation}\label{z3}
\nu_{\frak p}(z\bar z)=0,
\end{equation}
\begin{equation}\label{z4}
\nu_{\frak p}(z+\bar z)=
\begin{cases}
-p^m&\text{if}\ d'>1,\cr
-p^m+1&\text{if}\ d'=1,
\end{cases}
\end{equation}
(The proof of \eqref{z4} is similar to that of \eqref{vz+} and uses the assumption $n-d>1$ in the case $d'=1$.) Therefore,
\[
\nu_\frak P(z)+\nu_\frak P(\bar z)=0,
\]
\[
\nu_\frak P(z+\bar z)=\begin{cases}
-p^m&\text{if}\ d'>1,\cr
-p^m+1&\text{if}\ d'=1,
\end{cases}
\]
and it follows that
\[
\{\nu_\frak P(z),\nu_\frak P(\bar z)\}=\begin{cases}
\{\pm p^m\}&\text{if}\ d'>1,\cr
\{\pm(p^m-1)\}&\text{if}\ d'=1.
\end{cases}
\]
Since $z$ is $t$th power in $\overline\f_q(X,z)$, we have $t\mid\nu_\frak P(z)$. If $d'>1$, then $t=p$. It follows from $t\mid d$ and $t\mid n-d$ that $\text{gcd}(n,d)\ne 1$, which is a contradiction. So we must have $d'=1$ and $n-d=p^m$, $m>0$. Then $t\mid p^m-1=n-d-1$. Since $t\mid\text{gcd}(n-d-1,d)=\text{gcd}(n-1,d)$, where $\text{gcd}(n-1,d)$ is a power of 2 (by assumption), we have $t=2$. Consequently, $p$ is odd. 

The equation $Q_1(z)=0$ can be written as
\[
u^2=\delta,
\]
where
\[
u=z-\gamma,
\]
\[
\gamma=-\frac 12\frac{\displaystyle (a^{q+1}-1)\frac{X^d-1}{X-1}+(a^{q+1}X^d+1)(X-1)^{p^m-1}}{a(X-1)^{p^m-1}},
\]
and
\[
\delta=\gamma^2-a^{q-1}X^d.
\]
By assumption, there exist $\alpha,\beta\in\overline{\Bbb F}_q(X)$ such that
\[
(\alpha u+\beta)^2=u+\gamma,
\] 
i.e.,
\[
\alpha^2\delta+\beta^2+2\alpha\beta u=u+\gamma.
\]
So
\[
\begin{cases}
\alpha^2\delta+\beta^2=\gamma,\cr
2\alpha\beta=1.
\end{cases}
\]
Letting $\tau=\alpha/\beta$, we have
\begin{equation}\label{eq-tau1}
1+\delta\tau^2-2\gamma\tau=0
\end{equation}
and
\begin{equation}\label{tau1}
\tau=2\alpha^2.
\end{equation}
Equation~\eqref{eq-tau1} has an explicit solution
\[
\tau=\frac 1\delta(\gamma\pm a^{(q-1)/2}X^{d/2})=\frac 1{\gamma\mp a^{(q-1)/2}X^{d/2}}=\frac{-2a(X-1)^{p^m-1}}{h(X)},
\]
where
\[
h(X)=(a^{q+1}-1)\frac{X^d-1}{X-1}+(a^{q+1}X^d+1)(X-1)^{p^m-1}+2\epsilon X^{d/2}(X-1)^{p^m-1}
\]
and $\epsilon=\pm a^{(q+1)/2}$. By \eqref{tau1}, $h(X)$ is a square in $\overline\f_q[X]$, say $h=g^2$ for some $g\in\overline\f_q[X]$ with $\deg g=(d+p^m-1)/2$. Note that
\begin{align*}
&h(X)= a^{q+1}\Bigl(\frac{X^d-1}{X-1}+X^d\frac{X^{p^m}-1}{X-1} \Bigr)+\Bigl(\frac{X^{p^m}-1}{X-1}-\frac{X^d-1}{X-1} \Bigr)+2\epsilon X^{d/2}\frac{X^{p^m}-1}{X-1}\cr
&=a^{q+1}(1+\cdots+X^{p^m+d-1})+(X^d+\cdots+X^{p^m-1})+2\epsilon(X^{d/2}+\cdots+X^{p^m+d/2-1}),
\end{align*}
which is self-reciprocal. It follows that $g^*=\pm g$, where $g^*$ is the reciprocal polynomial of $g$. Let 
\[
H=(X-1)h=(a^{q+1}-1)(X^d-1)+(a^{q+1}X^d+1)(X-1)^{p^m}+2\epsilon X^{d/2}(X-1)^{p^m}.
\]
Then 
\[
H'=(a^{q+1}-1)dX^{d-1}+a^{q+1}dX^{d-1}(X-1)^{p^m}+\epsilon dX^{d/2-1}(X-1)^{p^m}.
\]
Let
\[
K=-H+d^{-1}XH'=a^{q+1}-X^{p^m}+\epsilon X^{d/2}-\epsilon X^{p^m+d/2}.
\]
The reciprocal of $K$ is
\[
K^*=-\epsilon+\epsilon X^{p^m}-X^{d/2}+a^{q+1}X^{p^m+d/2}.
\]
Since $g\mid K$, we have $g=\pm g^*\mid K^*$. Hence $g$ divides
\[
\epsilon K+K^*=(a^{q+1}-1)(\epsilon+X^{d/2}).
\]
This is a contradiction since
\[
\frac d2<\frac{d+p^m-1}2=\deg g.
\]

The proof of Theorem~\ref{T1.10} is now complete.

\subsection{Final Remarks}\

Theorem~\ref{T1.10} leaves ample room for improvement, by which we mean nonexistence results of PB under conditions that are weaker than or not covered by (i) -- (iii) in Theorem~\ref{T1.10}. While some improvements may be obtained by fine tuning the techniques demonstrated in the present paper, breakthroughs may require new methods or substantially new elements in the current approach.

The cases $d-n=\pm 1$ appear to be special. These are the two cases not covered by Theorem~\ref{T1.10} and there are indeed infinite classes of PBs in these two cases with $e=2$ (Results~\ref{R1.2} and \ref{R1.4}). A natural question is this: When $d-n=\pm 1$ and $e>2$, are there infinite classes of PBs of the form $f_{q,e,n,d,a}(X)=X^n(X^{d(q-1)}+a)$ of $\f_{q^e}$?
 
\section*{Acknowledgements}
	The research of Vincenzo Palozzi Lavorante was partially supported  by the Italian National Group for Algebraic and Geometric Structures and their Applications (GNSAGA - INdAM).



\end{document}